\documentclass{amsart}
\usepackage[latin2]{inputenc}
\usepackage{amsmath}
\usepackage{amsfonts}
\usepackage{pgfplots}
\usepackage{array,booktabs}
\newcolumntype{M}[1]{>{\centering\arraybackslash}m{#1}}
\usepackage{float}
\usepackage{enumitem}
\usepackage{collcell}

\usepackage{amssymb}
\usepackage[bb=ams,frak=mma,scr=zapfc]{mathalpha}

\pgfplotsset{compat=newest}
\usepackage{graphicx}
\pgfplotsset{width=7cm,compat=1.16} 
\usepackage{pgf,tikz,pgfplots}

\usepackage{mathrsfs}
\usetikzlibrary{arrows}

\usepackage{amsthm}
\usepackage[T1]{fontenc}
\usepackage{calc}
\usepackage{pstricks}
\usepackage{mathtools}
\usepackage{graphicx}
\usepackage{subcaption}
\usepackage{mciteplus}
\theoremstyle{definition}
\newtheorem*{remark}{Remark}
\usepackage[colorlinks]{hyperref}
\hypersetup{colorlinks,linkcolor={blue},citecolor={blue},urlcolor={red}} 
\theoremstyle{plain}

\newtheorem{lemma}{Lemma}
\theoremstyle{plain}

\newtheorem{theorem}{Theorem}
\newtheorem*{definition}{Definition}
\newtheorem*{remark3}{Acknowledgements}
\newcommand{\capac}{\operatorname{cap}}
\newcommand{\interior}{\operatorname{Int}}

\newcommand{\diam}{\operatorname{diam}}
\newcommand{\bd}{\operatorname{Bd}}
\DeclareMathOperator*{\esssup}{ess\,sup}
\DeclareMathOperator*{\essinf}{ess\,inf}

\usepackage[left=3.9cm, right=3.9cm, top=3.00cm, bottom=3.00cm]{geometry}
\title{Explosion Points and Topology of Julia Sets of Zorich maps}

\begin{document}
	\author{ATHANASIOS TSANTARIS}
\address{School of Mathematical Sciences, University of Nottingham, Nottingham NG7 2RD.}
\email{Athanasios.Tsantaris@Nottingham.ac.uk;}
	\maketitle
	\begin{abstract}
		Zorich maps are higher dimensional analogues of the complex exponential map. For the exponential family $\lambda e^z$, $\lambda>0$, it is known that for small values of $\lambda$ the Julia set is an uncountable collection of disjoint curves. The same was shown to hold for Zorich maps by Bergweiler and Nicks.
		
		 In this paper we introduce a topological model for the Julia sets of certain Zorich maps, similar to the so called \textit{straight brush} of Aarts and Oversteegen. As a corollary we show that $\infty$ is an \textit{explosion point} for the set of endpoints of the Julia sets. Moreover we introduce an object called a \textit{hairy surface} which is a compactified version of the Julia set of Zorich maps and we show that those objects are not uniquely embedded in $\mathbb{R}^3$, unlike the corresponding two dimensional objects which are all ambiently homeomorphic.  
	\end{abstract}
		\section{Introduction}
		\let\thefootnote\relax\footnote{2020 \textit{Mathematics Subject Classification}. Primary 30C65, 37F10; Secondary 54G15, 54F65}
	One of the most well-studied families of holomorphic dynamical systems is the exponential family $E_{\lambda}(z)=\lambda e^z$, $\lambda>0$. The primary object of study is the \textit{Julia set} $\mathcal{J}(E_\lambda)$, the set of points where the iterates $E_{\lambda}^n(z)$ do not form a normal family. It is well known (see the survey \cite{Devaney2010}) that when $0<\lambda<1/e$ the Julia set of $E_{\lambda}$ is a "\textit{Cantor Bouquet}". This roughly means that the Julia set consists of uncountably many disjoint curves  each connecting a point to infinity. There are a number of striking phenomena concerning the Julia sets of maps in the exponential family.  
	
	Recently there has been a systematic attempt to generalize results from the theory of iteration of holomorphic maps to the higher dimensional setting of quasiregular mappings. In \cite{berg2013, B-Nicks} Bergweiler and Nicks defined a Julia set for quasiregular maps in $\mathbb{R}^n$ that are analogous to rational or transcendental  maps in $\mathbb{C}$ and proved that this Julia set shares many of the properties of the classical Julia set. 
	
	In this higher dimensional setting there is a  map, first constructed by Zorich in \cite{Zorich}, called the Zorich map which can be thought of as the quasiregular version of the exponential map on the plane. In this paper we will restrict ourselves in $\mathbb{R}^3$ for simplicity. We will denote this map by $\mathcal{Z}$ and consider the family $\mathcal{Z}_{\lambda}=\lambda\mathcal{Z},$ $\lambda>0$. In \cite{bergk} (see also \cite[Section 7]{B-Nicks}) Bergweiler showed that the Julia set of $\mathcal{Z}_\lambda$, for small values of $\lambda$, has a similar structure with $\mathcal{J}(E_\lambda)$ for $\lambda<1/e$. In other words the Julia set consists of uncountably many disjoint curves each connecting a point to infinity.  It is also worth mentioning here that the dynamics of Zorich maps have been  studied for large values of the parameter $\lambda$. In \cite{Tsantaris2020} it was proven that for a slightly modified version of the Zorich maps and for large values of $\lambda$ the Julia set is $\mathbb{R}^3$. This can be seen as the analogous theorem to that of Misiurewicz for the exponential family in $\mathbb{C}$, see \cite{Misiu}.

	Concerning now the case of small values of the parameter $\lambda$, Bergweiler also proved that Zorich maps exhibit a dimension paradox, also known as Karpinska's paradox (see \cite{Karpinska,Karpinska1999}), which appears in the exponential family as well. This paradox concerns the endpoints of the curves that make up the Julia set. The striking fact is that while the Hausdorff dimension of the endpoints is 3  the curves without the endpoints have Hausdorff dimension 1.
	
	Another amazing fact, proven by Mayer in \cite{Mayer1990}, concerning the endpoints of those curves is that the point at $\infty$ is what is called a \textit{dispersion point} for the set of endpoints. In other words if we denote the set of endpoints by $\mathcal{E}\left(\mathcal{J}(E_\lambda)\right)$ then this set is totally disconnected while $\mathcal{E}\left(\mathcal{J}(E_\lambda)\right)\cup\{\infty\}$ is connected. That this amazing fact is even possible for a planar set is usually exhibited through a set known as Cantor's leaky tent or Knaster-Kuratowski fan (see for example \cite{LynnArthurSteen1995}). In fact something even stronger is true here,  the work of Aarts and Oversteegen in \cite{Aarts1993} shows that the point at $\infty$ is an \textit{explosion point}. This means that the set $\mathcal{E}\left(\mathcal{J}(E_\lambda)\right)$ is in fact totally separated, meaning that for any two of its points, $x,y$ there is a clopen subset $U\subset\mathcal{E}\left(\mathcal{J}(E_\lambda)\right)$ such that $x\in U$ but $y\not\in U$. Note that this property implies that $\mathcal{E}\left(\mathcal{J}(E_\lambda)\right)$ is totally disconnected and thus any explosion point is also a dispersion point. We also warn the reader that dispersion points have been sometimes called explosion points in the literature.
	
	  Moreover, it is worth mentioning that explosion and dispersion points for exponential maps have been studied for the subset of escaping points  of all endpoints  (endpoints that escape to $\infty$ under iteration) in \cite{Alhabib2016,Evdoridou2018}. Also, in \cite{comduhr2019} the presence of Cantor bouquets is proved for generalized exponential maps on the complex plane that are more general than Zorich maps (not necessarily quasiregular) and in fact those bouquets also have $\infty$ as an explosion point for their set of endpoints.\\
	
	 In the higher dimensional setting of Zorich maps now, Bergweiler in \cite{bergk} pointed out that $\infty$ might also be an explosion point for the set of endpoints  of the hairs of $\mathcal{J}(\mathcal{Z}_\lambda)$. One of the goals of this paper is to prove that this indeed holds.
	
Before we state our theorems let us define the Zorich map that we will work with. Although the construction can be made in arbitrary dimensions we confine ourselves in $\mathbb{R}^3$ for simplicity. First consider an $L$ bi-Lipschitz, sense-preserving map $h$ that maps the square \[Q:=\Big \{(x_1,x_2)\in\mathbb{R}^2:|x_1|\leq 1,|x_2|\leq 1\Big \}\] to the upper hemisphere \[\{(x_1,x_2,x_3)\in\mathbb{R}^3:x_1^2+x_2^2+x_3^2=1, x_3\geq 0\}.\] Then define $\mathcal{Z}:Q\times\mathbb{R}\to\mathbb{R}^3$ as \[\mathcal{Z}(x_1,x_2,x_3)=e^{x_3}h(x_1,x_2).\] The map $\mathcal{Z}$ maps the square beam $Q\times \mathbb{R}$ to the upper half-space. By repeatedly reflecting now, across the sides of the square beam and the $x_1x_2$ plane, we get a map $\mathcal{Z}:\mathbb{R}^3\to\mathbb{R}^3$. Note that this map is doubly periodic meaning that $\mathcal{Z}(x_1+4,x_2,x_3)=\mathcal{Z}(x_1,x_2+4,x_3)=\mathcal{Z}(x_1,x_2,x_3)$ and also $\mathcal{Z}\circ R=\mathcal Z$, where $R$ is a half-turn around the lines $x_1=2n+1, x_2=2m+1$, $n,m\in\mathbb{Z}$. Moreover, this map is not locally injective everywhere. The lines $x_1=2n+1, x_2=2m+1$, $n,m\in\mathbb{Z}$ belong to the \textit{branch set}, namely the set \[\mathcal{B}_{\mathcal{Z}}:=\{x\in\mathbb{R}^3:\mathcal{Z} \hspace{2mm}\text{is not locally homeomorphic at}\hspace{2mm}x\}.\] Also it can be shown that this map has an essential singularity at infinity, just like the exponential map on the plane, and it is quasiregular, see \cite{Iwaniec2001}. We call such quasiregular maps of \textit{transcendental type}. 

 We say that a subset $H$ of $\mathbb{C}$ (or $\mathbb{R}^n$) is a \textit{hair} if there exists a homeomorphism $\gamma~:~[0,\infty) \to H$ such that $\gamma(t)\to\infty$ as
$t\to\infty$. We call $\gamma(0)$ the endpoint of the hair $H$.

Consider now the family $\lambda \mathcal{Z}$, $\lambda>0$.
\begin{theorem}[Bergweiler and Nicks, \cite{bergk,B-Nicks}]\label{berg1}\hspace{1mm}\\
	For small enough values of $\lambda>0$ the Julia set $\mathcal{J}(\mathcal{Z}_\lambda)$ of a map in the Zorich family consists of uncountably many disjoint hairs. Each of the hairs tends to $\infty$ by having their third coordinate go to $\infty$ while the other two coordinates remain bounded. Moreover, for each point $x\not\in\mathcal{J}(\mathcal{Z}_\lambda)$, $\mathcal{Z}^n_\lambda(x)$ converges to a fixed point.
\end{theorem}
\begin{remark}
	In \cite{bergk,B-Nicks} Theorem \ref{berg1} is shown to be true for small enough values of $\lambda$ without stating an explicit estimate for suitable values of $\lambda$. In section 2 we make this more precise by proving that Theorem \ref{berg1} holds for $0<\lambda<e^{-\left(\log L+L\right)}$.
\end{remark}
We now state our result on explosion points.
\begin{theorem}\label{main}
	Let $0<\lambda<e^{-\left(\log L+L\right)}$. Then $\infty$ is an explosion point for the set of endpoints of hairs in the Julia set, $\mathcal{E}(\mathcal{J}(\mathcal{Z}_\lambda))$.
\end{theorem}

In the process of proving the above theorem we will need to introduce a topological model for  $\mathcal{J}(\mathcal{Z}_\lambda)$. That model is a generalization, in three dimensions, of a model that Aarts and Oversteegen introduced in \cite{Aarts1993} in order to study the topology of Julia sets. They called their model a \textit{straight brush} and proved that the Julia set of the exponential map $E_\lambda$, $0<\lambda<1/e$ is homeomorphic to a straight brush. We call our new three dimensional model a \textit{3-d straight brush}  (see section 2 for details) and we will prove the following.

\begin{theorem}\label{brush}
	Let $0<\lambda<e^{-\left(\log L+L\right)}$. Then there is a 3-d straight brush $B$ and a homeomorphism of $B$ onto the Julia set $\mathcal{J}(\mathcal{Z}_\lambda)$ of the Zorich map. Moreover, this homeomorphism extends to a homeomorphism between $B\cup\{\infty\}$ and $\mathcal{J}(\mathcal{Z}_\lambda)\cup\{\infty\}$. 
\end{theorem}

Another object that Aarts and Oversteegen introduced in the same paper was the \textit{straight one-sided hairy arc} which is compact object living in $[0,1]^2$. It turns out that  if we suitably embed a straight brush in the square $[0,1]^2$ and then compactify we get such an object. It also turns out that any two such objects are ambiently homeomorphic in the complex plane so that there is homeomorphism between them that also extends to the entire complex plane. Moreover,  if we compactify $\mathcal{J}(E_\lambda)$, $0<\lambda<1/e$ in a certain way then we get a space which is homeomorphic to a straight one-sided hairy arc. Aarts and Oversteegen call the homeomorphic images of straight one-sided hairy arcs just \textit{hairy arcs}.  Moreover, they show that if we embed a hairy arc to the plane in a way that it has the extra property of one-sidedness then it is ambiently homeomorphic, this time the ambient space being the Riemann sphere,  with a straight one-sided hairy arc. As a result the compactified version of $\mathcal{J}(E_\lambda)$ is a hairy arc which when suitably embedded in the Riemann sphere is ambiently homeomorphic to a  straight one-sided hairy arc.

Such considerations make sense in higher dimensions too. Thus in section \ref{hairysurfaces} we define a \textit{straight one-sided hairy square} and a \textit{hairy surface} which are the analogous objects to straight one-sided hairy arcs and hairy arcs respectively (see section \ref{hairysurfaces} for details). 

Similarly with the case of exponential maps on the plane we can show 

\begin{theorem}\label{compactify}
	Let $0<\lambda<e^{-\left(\log L+L\right)}$. Then there is a compactification $\widetilde{\mathcal{J}(\mathcal{Z}_\lambda)}$ of $\mathcal{J}(\mathcal{Z}_\lambda)$ which is a hairy surface.
\end{theorem}

Two other natural questions to ask in this higher dimensional setting would be whether or not any two straight one-sided hairy squares are ambiently homeomorphic (the ambient space being $\mathbb{R}^3$) and whether all one-sided hairy surfaces are ambiently homeomorphic to a straight-one sided hairy square. Both of these are true in the two dimensional setting as we have already mentioned.

Although we will not give an answer to the first question here, we can show that in higher dimensions not all one-sided hairy surfaces are ambiently homeomorphic to a straight one-sided hairy square as the next theorem shows.

\begin{theorem}\label{wild}
	There exists a straight one-sided hairy square $S$ and a homemorphism $$H:S\to H(S)\subset\mathbb{R}^3$$ such that $H(S)$ is one sided and $H(S)$ and $S$ are not ambiently homeomorphic, that is there is no homeomorphism of $\mathbb{R}^3$ that maps $S$ to $H(S)$.
\end{theorem}

The structure of the rest of paper is as follows. In section 2 we give some definitions, some preliminary results and the general idea for the proof of Theorem \ref{main}. In section 3 we construct the 3-d straight brush that corresponds to the Julia set and we prove Theorem~\ref{brush}. In section 4 we prove that $\mathcal{J}(\mathcal{Z}_\nu)\cup\{\infty\}$ is a Lelek fan (see section \ref{section2} for the definition). In Section 5 we define straight one sided hairy squares and we prove Theorem \ref{compactify}.  Finally, in section 6 we discuss the proof of Theorem \ref{wild}.

\begin{remark3}
	\emph{I would like to thank my supervisor,  Daniel Nicks, for all his help and encouragement while writing this paper.}
\end{remark3}

\section{Main Idea and preliminaries for the proof of Theorem \ref{main}}\label{section2}
Let us briefly mention what quasiregular maps are. For the general theory of quasiregular maps we refer to \cite{Rickman} and \cite{vuorinen}.   Also, see \cite{Berg1} for a survey on quasiregular dynamics.

If $d\geq 2$ and $G\subset \mathbb{R}^d$ is a domain, then for $1\leq p <\infty$ the  \textit{Sobolev space} $W^1_{p,loc}(G)$ consists of functions $f=(f_1,f_2,\cdots, f_d):G\to\mathbb{R}^d$ for which the first order weak partial derivatives $\partial_i f_j$ exist and are locally in $L^p$. A continuous map $f\in W^1_{d,loc}(G)$ is called \textit{quasiregular} if there exists a constant $K_O\geq 1$ such that \begin{equation}\label{quasi}
\left|Df(x)\right|^d\leq K_O J_f(x) \hspace{2mm} a.e.,\end{equation} where $Df(x)$ denotes the total derivative at point $x$, $$|Df(x)|=\sup_{|h|=1}|Df(x)(h)|$$ denotes the operator norm of the derivative, and $J_f(x)$ denotes the Jacobian determinant. Also let $$\ell(Df(x))=\inf_{|h|=1}|Df(x)(h)|.$$
The condition that \eqref{quasi} is satisfied for some $K_O\geq1$ implies that $$K_I\ell(Df(x))^d\geq J_f(x),\hspace{2mm} a.e.,$$
for some $K_I\geq 1$. The smallest constants $K_O$ and $K_I$ for which those two conditions hold are called the \textit{outer dilatation} and \textit{inner dilatation} respectively. We call the maximum of those two numbers the dilatation of $f$ and we denote it by $K(f)$.\\

 In \cite{berg2013} Bergweiler developed a Fatou-Julia theory for  quasiregular self-maps of $\overline{\mathbb{R}^d}$ which  can be thought of as analogues of rational maps, while in \cite{B-Nicks} Bergweiler and Nicks did the same but for transcendental type quasiregular maps. To define a Julia set for such maps we need the notion of \textit{conformal capacity} for which we refer to \cite[Chapter II, Section 10]{Rickman}. Following those two papers we define the Julia set of $f:\mathbb{R}^d\to \mathbb{R}^d$, denoted ${\mathcal{J}}(f)$, to be the set of all those $x\in\mathbb{R}^d$ such that
 $$ \capac \left(\mathbb{R}^d\setminus \bigcup_{k=1}^\infty f^k(U)\right)=0
 $$	for every neighbourhood $U$ of $x$, where $\capac$ denotes the conformal capacity. Following \cite{NICKS2017} we call the complement of ${\mathcal{J}}(f)$ the \textit{quasi-Fatou set}, and we denote it by $QF(f)$.  
 
  It turns out that the Julia set we defined enjoys many of the properties that the classical Julia set for holomorphic maps has. In particular, a property that we will use in this paper is that the Julia set is a \textit{completely invariant} set, meaning that $x\in\mathcal{J}(f)$ if and only if $f(x)\in \mathcal{J}(f)$. For more details and motivation behind the definition of the Julia set we refer to \cite{Berg1,berg2013,B-Nicks}. 
 
 As we already have noted, the Zorich map is a quasiregular map of transcendental type and thus the above definitions make sense for this map. \\

 The main idea of the proof of Theorem \ref{main} is to first show that the Julia set $\mathcal{J}(\mathcal{Z}_\lambda)$ is a so called \textit{Lelek fan}. Let us first define Lelek fans before we explain why we need this. 
 \begin{definition}[Fans]
 	Let $X$ be a continuum (compact and connected metric space). Then $X$ is a fan with \textit{top} $x_0$ when the following conditions are satisfied.
 	\begin{enumerate}[label=(\roman*)]
 		\item $X$ is hereditarily unicoherent, meaning that $K\cap L$ is connected for every pair of subcontinua  $K$, $L$.
 		\item $X$ is arcwise connected. This together with (i) implies that $X$ is uniquely arcwise connected.
 		\item $x_0$ is the only point that is the common endpoint of at least three different arcs that are otherwise disjoint.
 	\end{enumerate}
 \end{definition}

 Let $X$ be a fan and  $x$, $y\in X$. With $[x,y]$ we  will denote the unique arc connecting $x$ and $y$. Also if a point $x\in X$ is an endpoint of every arc in $X$ containing it then we call this point an \textit{endpoint}. We will use the same symbol as in the introduction to denote the set of endpoints $\mathcal{E}(X)$ of a fan.

 \begin{definition}[Lelek fans]
 	A fan with top $x_0$ is called a Lelek fan if it has the following two properties.
 	\begin{enumerate}[label=(\roman*)]
 		\item \textbf{Smoothness:} For any sequence $y_n\in X$ converging to $y\in X$ the arcs $[x_0,y_n]$ converge to $[x_0,y]$ in the Hausdorff metric.
 		\item \textbf{Density of endpoints:} The endpoints of $X$ are dense in $X$.
 	\end{enumerate} 
 \end{definition}
With this terminology we have the following.
\begin{theorem}\label{fan}
	Let $0<\lambda<e^{-\left(\log L+L\right)}$ be such that Theorem \ref{berg1} holds. Then $\mathcal{J}(\mathcal{Z}_\lambda)\cup\{\infty\}$ is a Lelek fan with top at $\infty$.
\end{theorem}

The reason why this theorem is important for what we want to prove is that 
Lelek in \cite{Lelek1961} gave an example of a Lelek fan with top $\{x_0\}$ and showed that $x_0$ is an explosion point for the set of endpoints $\mathcal{E}(X)$. Much later
Charatonik and independently Bula and Oversteegen  proved the following.
\begin{theorem}[Charatonik, Bula-Oversteegen \cite{Bula1990,char1989}]\label{bulas}\hspace{1mm}\\
	Any two Lelek fans are homeomorphic.
\end{theorem}

The above theorem combined with Theorem \ref{fan} allows us to prove Theorem \ref{main}.
\begin{proof}[Proof of Theorem \ref{main}]
	Since, by Theorem \ref{fan}, $\mathcal{J}(\mathcal{Z}_\lambda)\cup \{\infty\}$ is a Lelek fan it follows by Theorem \ref{bulas} that it is homeomorphic to the one that Lelek constructed. Hence, because endpoints get mapped to endpoints and the top gets mapped to the top we get that  $\mathcal{E}(\mathcal{J}(\mathcal{Z}_\lambda))\cup\{\infty\}$ is connected while $\mathcal{E}(\mathcal{J}(\mathcal{Z}_\lambda))$ is totally separated.
\end{proof}
So the only thing left to prove now is Theorem \ref{fan}. For that we will need to construct a topological model for the Julia set for which it will be easier to show that it is a Lelek fan. For the exponential map a topological model for the Julia set is given by the so called \textit{straight brush} which was first introduced by Aarts and Oversteegen in \cite{Aarts1993}. Following them, we define a three dimensional version of a straight brush.    
	\begin{definition}[3-d Straight Brush]\hspace{1mm}\\
		A 3-d Straight Brush $B$ is a subset of \[\{(y,a_1,a_2)\in\mathbb{R}^3:y\geq0, (a_1,a_2)\in\left(\mathbb{R}\setminus\mathbb{Q}\right)^2\}\] with the following properties:\begin{enumerate}[label=(\roman*)]
			\item \textbf{Hairiness:} For every $(a_1,a_2)\in \mathbb{R}^2$, there is a $t_{(a_1,a_2)}\in[0,\infty]$ such that $(t,a_1,a_2)\in B$ if and only if $t\geq t_{(a_1,a_2)}$
			\item\textbf{Density:} The set of $(a_1,a_2)$ with $t_{(a_1,a_2)}<\infty$ is dense in $\left(\mathbb{R}\setminus\mathbb{Q}\right)^2$. Also, for any such $(a_1,a_2)$ there exist  sequences $(a_{1},a_{n,2})$, $(a_{1},b_{n,2})$, $(c_{n,1},a_{2})$, $(d_{n,1},a_{2})$,  such that $ a_{n,2}\uparrow a_2$, $ b_{n,2}\downarrow a_2$, $c_{n,1}\uparrow a_1$, $d_{n,1}\downarrow a_1$. Moreover it is true that $t_{(a_{1},a_{n,2})}\to t_{(a_1,a_2)}$ and similarly for the other sequences. 
			\item \textbf{Compact Sections:} $B$ is a closed subset of $\mathbb{R}^3$.
		\end{enumerate}
		
	\end{definition}

In order to prove Theorem \ref{fan} we will need to prove Theorem \ref{brush} first.

 Here let us also introduce some notation that we will need later. For $(r_1,r_2)\in\mathbb{Z}^2$ we  set \[P(r_1,r_2):=\{(x_1,x_2)\in\mathbb{R}^2:|x_1-2r_1|<1,|x_2-2r_2|<1\}.\] For any $c\in\mathbb{R}$ we also define the half space \[H_{>c}:=\{(x_1,x_2,x_3)\in\mathbb{R}^3:x_3>c\}.\] Note now that $\mathcal{Z}_\lambda$ maps $P(r_1,r_2)\times \mathbb{R}$ bijectively onto $H_{>0}$ or $H_{<0}$ depending on whether $r_1+r_2$ is even or odd. Also we mention here that from Bergweiler's and Nicks' work in \cite{bergk,B-Nicks} it follows that for the values of $\lambda$ they consider
 
  \begin{equation}\label{eqfatou}
  	 \overline{P(r_1,r_2)}\hspace{1mm}\text{belongs to the quasi-Fatou set when}\hspace{1mm} r_1+r_2=\text{odd}\hspace{1mm}. \end{equation}
 
 We will also need the notion of the \textit{itinerary} of a point $x\in\mathcal{J}(\mathcal{Z}_\lambda)$. To each such point we can associate a sequence \[\Delta(x)=n_0n_1n_2\dots,\] where $n_k=(n_{k,1},n_{k,2})\in\mathbb{Z}\times\mathbb{Z}$ and $n_{k,1}+n_{k,2}=\text{even}$, in such a way that $$\mathcal{Z}^k_\lambda(x)\in P(n_k)\times \mathbb{R},$$ for all $k\in\mathbb{N}$. That sequence we will call the itinerary of $x$.  
 
 We also set $p:\mathbb{R}^3\to \mathbb{R}^2$, $p_3\mathbb{R}^3\to\mathbb{R}$ to be the projection maps defined by $$p(x_1,x_2,x_3)=(x_1,x_2) \hspace{2mm}\text{and}\hspace{2mm}  p_3(x_1,x_2,x_3)=x_3.$$

 Let us note here again that in \cite{bergk} it is shown that Theorem \ref{berg1} holds for all sufficiently small values of the parameter $\lambda$ without an explicit estimate for those values. In what follows we make this more precise.
 
 \begin{lemma}\label{lemmaeq}
 	\begin{equation}\label{eq111}\frac{\lambda e^{x_3}}{L}\leq \ell\left(D\mathcal{Z}_\lambda(x_1,x_2,x_3)\right)\leq \left|D\mathcal{Z}_\lambda(x_1,x_2,x_3)\right|\leq \lambda Le^{x_3}\hspace{2mm }\text{a.e.}\end{equation}
 \end{lemma}
\begin{proof}
	We assume that $(x_1,x_2)\in Q$ since the other case can be handled similarly.
	We note that \begin{equation}\label{zor}D\mathcal{Z}_\lambda(x)=e^{x_3}D\mathcal{Z}_\lambda(x_1,x_2,0).\end{equation}
	 Then  $D\mathcal{Z}_\lambda(x_1,x_2,0)$ is the linear map induced by the matrix \[\lambda\begin{pmatrix}
		\frac{\partial h_1}{\partial x_1}(p(x))&\frac{\partial h_1}{\partial x_2}(p(x))&h_1(p(x))\\\\\frac{\partial h_2}{\partial x_1}(p(x))&\frac{\partial h_2}{\partial x_2}(p(x))&h_2(p(x))\\\\\frac{\partial h_3}{\partial x_1}(p(x))&\frac{\partial h_3}{\partial x_2}(p(x))&h_3(p(x))
	\end{pmatrix},\]
	where $h=(h_1,h_2,h_3)$ is the bi-Lipschitz map we used in the construction of the Zorich map.

	We now set \[A=\begin{pmatrix}
		\frac{\partial h_1}{\partial x_1}\left(p(x)\right)\\\\\frac{\partial h_2}{\partial x_1}\left(p(x)\right)\\\\\frac{\partial h_3}{\partial x_1}\left(p(x)\right)
	\end{pmatrix},\hspace{2mm} B=\begin{pmatrix}
		\frac{\partial h_1}{\partial x_2}\left(p(x)\right)\\\\\frac{\partial h_2}{\partial x_2}\left(p(x)\right)\\\\\frac{\partial h_3}{\partial x_2}\left(p(x)\right)\end{pmatrix},\hspace{2mm} C=\begin{pmatrix}
		h_1\left(p(x)\right)\\\\h_2\left(p(x)\right)\\\\h_3\left(p(x)\right)
	\end{pmatrix}
	\]
	
	and $v=(v_1,v_2,v_3)\in\mathbb{R}^3$. Then by using \eqref{zor}
	\begin{align*}\ell\left(D\mathcal{Z}_{\lambda}(x)\right)^2=\inf_{|v|=1}|D\mathcal{Z}_{\lambda}(x)(v)|^2= e^{2x_3}\lambda^2\inf_{|v|=1}|v_1A+v_2B+v_3C|^2. \end{align*}
	
	Notice now that $C$ is orthogonal to $A$ and $B$ and thus the  above equation becomes
	\begin{align*}
		\ell\left(D\mathcal{Z}_{\lambda}(x)\right)^2=&e^{2x_3}\lambda^2\inf_{|v|=1}\left(|v_1A+v_2B|^2+|v_3C|^2\right)\\=&e^{2x_3}\lambda^2\inf_{|v|=1}\left(\left|Dh(p(x))(v_1,v_2)\right|^2+|v_3|^2|C|^2\right)\\\geq&e^{2x_3}\lambda^2\inf_{|v|=1}\left(\frac{1}{L^2}|(v_1,v_2)|^2+|v_3|^2\right)\\\geq&e^{2x_3}\frac{\lambda^2}{L^2}\hspace{2mm}\text{a.e.}
	\end{align*}
Similarly we can show that $\left|D\mathcal{Z}_\lambda(x_1,x_2,x_3)\right|^2 \leq \lambda^2 L^2 e^{2x_3}$ so that \eqref{eq111} readily follows.
\end{proof}

 \begin{lemma}\label{lemmalam}
 	For all $0<\lambda<e^{-\left(\log L+L\right)}$ Theorem \ref{berg1} is true.
 \end{lemma}
\begin{proof}
	First we note that in \cite{bergk,B-Nicks}  a different parametrization for the Zorich family is used, namely $\mathscr{Z}_{\kappa}(x)=\mathcal{Z}(x)+(0,0,\kappa$), $\kappa\in\mathbb{R}$. Moreover, Theorem \ref{berg1} is shown to hold for all $$\kappa\leq M_2-e^{M_1},$$ where $M_1:=M_1(\alpha)>0$ is such that $x_3\geq M_1 \hspace{2mm}\text{implies}\hspace{2mm} \ell\left(D\mathcal{Z}(x_1,x_2,x_3)\right)\geq \frac{1}{\alpha} $, $M_2:=M_2(\alpha)$ is such that $x_3\leq M_2$ implies that $\left|D\mathcal{Z}(x_1,x_2,x_3)\right|\leq \alpha$ and $\alpha$ is any number in $(0,1)$ (see \cite[(1.4),(1.5)]{bergk}). It is easy to see now using Lemma \ref{lemmaeq}, for $\lambda=1$, that we can take  $$M_1= \log\frac{L}{\alpha}\hspace{2mm} \text{and} \hspace{2mm}M_2= \log\frac{\alpha}{L}.$$

	Of course the two parametrizations are conjugate with $\kappa=\log \lambda$. Converting  the facts of the above paragraph in our setting we get that Theorem \ref{berg1} holds for $$0<\lambda\leq e^{M_2-e^{M_1}}=e^{\log\frac{\alpha}{L}-\frac{L}{\alpha}}.$$ Taking $\alpha\to 1$ we get that Theorem \ref{berg1} is true for all $0<\lambda<e^{-\left(\log L+L\right)}$.
\end{proof}

 Later we will also need the fact that the Zorich maps are expansive in a suitable upper half space. This is analogous to the fact that the exponential map is expansive in a right half plane. The next lemma makes this precise.
 
 \begin{lemma}\label{inverseexpansion}
 	For any $0<\lambda<e^{-\left(\log L+L\right)}$ there is an  $0<\alpha<1$ such that if $M=\log\frac{L}{\lambda\alpha}$ and
 	$\Lambda:H_{\geq M}\to P(r_1,r_2)$, for some $(r_1,r_2)\in\mathbb{Z}^2$,  is an inverse branch of the Zorich map $\mathcal{Z}_\lambda$ defined in $H_{\geq M}$, then $ \mathcal{Z}_\lambda(H_{\leq M})\subset H_{<M}$ and for all  $x,y\in H_{\geq M}$    \begin{equation}\label{eqlemma1}|\Lambda(x)-\Lambda(y)|\leq \alpha|x-y|.\end{equation} 
 \end{lemma}
\begin{proof}
		First note that \[D\Lambda(x)=D\mathcal{Z}_\lambda(\Lambda(x))^{-1}.\] 
		The Zorich map is absolutely continuous on any line segment since it is  locally Lipschitz. Using the fundamental theorem of calculus for the Lebesgue integral now, it is not too hard to prove  that a version of  the finite increment theorem, see \cite[10.4.1, Theorem 1]{Zorich2019}, is true for such functions. In other words, if $\gamma$ is the segment that connects $x$ and $y$ then
	\begin{equation}\label{eq10}|\Lambda(x)-\Lambda(y)|\leq\esssup_{z\in \gamma}|D\Lambda(z)||x-y| .\end{equation}
		
		It is true that \begin{equation}\label{eq12}\esssup_{z\in \gamma}|D\Lambda(z)|\leq \frac{1}{\essinf_{z\in\gamma}\ell\left(D\mathcal{Z}_\lambda(\Lambda(z))\right)}.\end{equation}
		By Lemma \ref{lemmaeq} we have that 
		 \begin{equation}\label{eqell}\ell\left(D\mathcal{Z}_{\lambda}(\Lambda(x))\right)\geq \frac{e^{p_3(\Lambda(x))}\lambda}{L}\hspace{2mm}\text{a.e.}\end{equation}

		 Hence for any $0<\alpha<1$ and for all $x$ such that \begin{equation}\label{eqexpansion}p_3(\Lambda(x))\geq \log \frac{L}{\lambda\alpha}\end{equation} 
		 we have  $\ell(D\mathcal{Z}_\lambda(\Lambda(x)))\geq1/\alpha$.
		 
		 We now claim that for  $M=\log\frac{L}{\lambda\alpha}$ we have that for all $x\in H_{\geq M}$ equation \eqref{eqexpansion} holds. Indeed if $x=(x_1,x_2,x_3)\in H_{\leq M}$ we have that

		  \begin{equation}\label{eqhalf}p_3\left(\mathcal{Z}_\lambda(x)\right)=p_3\left(\lambda e^{x_3}h(x_1,x_2)\right)\leq \lambda\frac{L}{\lambda\alpha}=\frac{L}{\alpha}.\end{equation} 
		 Notice now that by continuity we can choose $0<\alpha<1$ so that \[\lambda< e^{-\left(\log L+L\right)}=\frac{e^{-L}}{L}\leq \frac{e^{-\frac{L}{\alpha}}L}{\alpha},\]
		 which after rearranging and using \eqref{eqhalf} implies that 
		 \[p_3\left(\mathcal{Z}_\lambda(x)\right)<\log \frac{L}{\lambda\alpha}.\]
		  Hence $ \mathcal{Z}_\lambda(H_{\leq M})\subset H_{<M}$ as we wanted. Moreover, this implies that whenever $x\in H_{\geq M}$ we have that $\Lambda(x)\in H_{\geq M}$ so that \eqref{eqexpansion} is true.

		  Putting everything together we have that for $x\in H_{\geq M}$ \[|D\Lambda(x)|\leq\frac{1}{\ell(D\mathcal{Z}_\lambda(\Lambda(x)))}\leq\alpha\hspace{2mm}\text{a.e.} \] Combined with (\ref{eq10}) and (\ref{eq12}) this gives us what we wanted.
	\end{proof}

The next lemma introduces a number $p_\lambda$ which is going to play an important role in the subsequent sections.

\begin{lemma}\label{pl-lemma}
	Let $0<\lambda<e^{-\left(\log L+L\right)}$ and $\alpha$, $M$ be as in Lemma \ref{inverseexpansion}. Then there is a number $p_\lambda>1$ such that \eqref{eqlemma1} holds for all $x,y\in H_{\geq p_\lambda}$ and such that all planes $x_3=c$, where $c>p_\lambda+1$, intersect the Julia set $\mathcal{J}(\mathcal{Z_\lambda})$.
\end{lemma}
\begin{proof}

From Lemma \ref{inverseexpansion} we know  that $\mathcal{Z}_\lambda(H_{\leq M})\subset H_{<M}$ so that $H_{\leq M}$ is in the quasi-Fatou set. Moreover,  Theorem \ref{berg1} implies that there is  a minimum number $J$ such that for all $c\geq J$ the plane $\{(x_1,x_2,x_3):x_3=c\}$ intersects the Julia set $\mathcal{J}(\mathcal{Z}_\lambda)$ and thus $J\geq M$. Also all the points in $H_{<J}$ are in the quasi-Fatou set and converge to a fixed point.

  We set 
\begin{equation*}\label{eq006}
p_\lambda=\max\{J-1,M\}.
\end{equation*}
Also notice that  \begin{equation*}\label{eqpl}p_\lambda\geq M=\log\frac{L}{\lambda\alpha}\geq \log\frac{L}{e^{-(\log L+ L)}\alpha}\geq 2\log L+L-\log \alpha>1,\end{equation*}
where in the last inequality we used the fact that $L> 1$ and $0<\alpha<1$.
\end{proof}

\section{Proving Theorem \ref{brush}}\label{section3}
Our proof will closely follow that of the corresponding result for the exponential map which is due to Aarts and Oversteegen (see \cite[Theorem 1.4]{Aarts1993}).
\subsection{Construction of the 3-d straight brush }
First we will need to define a correspondence between the set $ \left(\mathbb{R}\setminus\mathbb{Q}\right)^2$ and $\prod_{i=0}^{\infty}\mathbb{Z}\times \mathbb{Z}$. This can be done in many ways but let us mention here a method used by Devaney based on Farey trees. In \cite[Section 5.3]{Devaney2010} Devaney finds for every irrational number $\zeta$ a sequence of integers $n_0n_1n_2\dots$ by doing the following. We break the real line into intervals $I_k=(k,k+1)$ for each $k\in\mathbb{Z}$. Then we further subdivide each $I_k$ in intervals $I_{kl}$, $l\in\mathbb{Z}$ in a certain way and so on. Specifically, assuming that $I_{n_0n_1\dots n_k}$ has been defined we define $I_{n_0\dots n_k j} $ as follows.
Let $$I_{n_0n_1\dots n_k}=\left(\frac{a}{b},\frac{c}{d}\right)$$ and $\frac{p_0}{q_0}=\frac{a+c}{b+d}$. We then define $\frac{p_n}{q_n}=\frac{p_{n-1}+c}{q_{n-1}+d}$ and $\frac{p_{-n}}{q_{-n}}=\frac{p_{-n+1}+a}{q_{-n+1}+b}$, for all $n\in\mathbb{N}$. Finally, define \[I_{n_0\dots n_k j}=\left(\frac{p_j}{q_j},\frac{p_{j+1}}{q_{j+1}}\right).\]

  We refer to \cite{Devaney2010} for more details but the whole construction implies that the intervals we get satisfy the following.\begin{enumerate}
	\item $I_{n_0n_1\dots n_{k+1}}\subset I_{n_0n_1\dots n_k}.$
	\item The endpoints of each interval $I_{n_0n_1\dots n_k}$ are rational.
	\item $\{\zeta\}=\bigcap_{k=1}^{\infty} I_{n_0n_1\dots n_k}$ and every irrational can be obtained this way.
\end{enumerate}  
In our case for any pair of irrational numbers $(a_1,a_2)$  we can apply this method twice and get a pair of sequences $\{(n_{k,1},n_{k,2})\}_{k\in\mathbb{N}}\in \prod_{i=0}^{\infty}\mathbb{Z}\times \mathbb{Z}$. We also equip $\prod_{i=0}^{\infty}\mathbb{Z}\times \mathbb{Z}$ with the product topology of the spaces $\mathbb{Z}\times \mathbb{Z}$ equipped with the discrete topology. The correspondence between  $ \left(\mathbb{R}\setminus\mathbb{Q}\right)^2$ and $\prod_{i=0}^{\infty}\mathbb{Z}\times \mathbb{Z}$ can be shown to be a homeomorphism.

Let now $x\in[0,\infty)$ and $(s_1,s_2)\in\mathbb{Z}^2$ with $s_1+s_2=\text{even}$. We define the cubes \[S(x,s_1,s_2):=\{(x_1,x_2,x_3)\in\mathbb{R}^3:x\leq x_3\leq x+1, (x_1,x_2)\in \overline{P(s_1,s_2)}\}.\]
Note that because the Zorich map is doubly periodic with periods $(4,0,0)$, $(0,4,0)$ and because $\mathcal{Z}\circ R=\mathcal{Z}$, where $R$ is a half-turn around the lines $x_1=2n+1, x_2=2m+1$, $n,m\in\mathbb{Z}$ the Julia set will be periodic and invariant under $R$ as well. Also note that by \eqref{eqfatou} the Julia set lies in the square beams $P(s_1,s_2)\times\mathbb{R}$ with $s_1+s_2=$ even and thus   when $s_1+s_2$ is even

\begin{equation}\label{juliasquare}S(x,s_1,s_2)\cap \mathcal{J}(\mathcal{Z}_\lambda)\not=\emptyset,\hspace{1mm}\text{for all}\hspace{2mm} x\geq p_\lambda,\end{equation}
  where $p_\lambda$ was defined in Lemma \ref{pl-lemma}.

Let us now construct the 3-d Straight Brush $B$ of Theorem \ref{brush} which will be homeomorphic to our Julia set. We will suitably modify the construction done in \cite{Aarts1993}. Let $x\geq p_\lambda$ and $(a_1,a_2)\in\left(\mathbb{R}\setminus\mathbb{Q}\right)^2$ with the corresponding sequence  $\{(n_{k,1},n_{k,2})\}_{k\in\mathbb{N}\cup\{0\}}$. Set $x_0=x$ and $R_0(x,a_1,a_2)=S(x,n_{0,1},n_{0,2})$. By induction on $k$  now we define $R_k=R_k(x,a_1,a_2)$ and $x_k$. We consider now two cases:
\begin{enumerate}[label=(\roman*)]
	\item $R_k\not=\emptyset$ and there is a $\xi$ with \begin{equation}\label{eq1}S(\xi,n_{k+1,1},n_{k+1,2})\subset \mathcal{Z}_{\lambda}\left(R_k\right).\end{equation}
	Let $\xi'=\min\{\xi:\xi \hspace{1mm}\text{satisfies}\hspace{1mm}(\ref{eq1})\}$. If $\xi'\geq p_\lambda$, we set $x_{k+1}=\xi'$ and \[R_{k+1}=S(\xi',n_{k+1,1},n_{k+1,2}).\]
	If $\xi'< p_\lambda$, we set $x_{k+1}=x_k$ and $R_{k+1}=\emptyset$.
	\item If the  case (i) does not apply, we set $x_{k+1}=x_k$ and $R_{k+1}=\emptyset$.
	\end{enumerate}

 For each $k\in \mathbb{N}$ now we set $$B_k=B_k(x):=\{y\in R_0:\mathcal{Z}^j_\lambda(y)\in R_j, \hspace{2mm}\text{for}\hspace{2mm}1\leq j\leq k\}.$$ The 3-d straight brush is then defined as \[B:=\{(x,a_1,a_2):R_k\not=\emptyset, \hspace{1mm}\text{for all}\hspace{1mm}k\in\mathbb{N}\}.\]
 
 The proof that this construction defines a 3-d straight brush and that this is homeomorphic to the Julia set will be now split in several lemmas. First we need to show that the set $B$ we defined is indeed a 3-d straight brush. 
\begin{lemma}
$B$ is a 3-d straight brush.
\end{lemma}
\begin{proof}First we prove hairiness.\\
	
	\textbf{\underline{Hairiness.}}\\Let $(x,a_1,a_2)\in B$ with $\{(n_{k,1},n_{k,2})\}_{k\in\mathbb{N}}$ the corresponding sequence of $(a_1,a_2)$. Suppose that $x<y$ then $(y,a_1,a_2)\in B$. Indeed, by induction $x_k<y_k$ and $R_k(y,a_1,a_2)\not=\emptyset$ for all $k\in\mathbb{N}$. Hence $(y,a_1,a_2)\in B$. Suppose now that $(z,a_1,a_2)\not\in B$. Consider the smallest $k$ with $R_k(z,a_1,a_2)\not=\emptyset$ but $R_{k+1}(z,a_1,a_2)=\emptyset$. Then either there is no $\xi$ with \[S(\xi,n_{k+1,1},n_{k+1,2})\subset \mathcal{Z}_\lambda(R_k(z,a_1,a_2))\] or there is such a $\xi$ but for the minimal $\xi_m$ we have that $\xi_m<p_\lambda$. In both cases the same holds for all $y$ slightly larger than $z$. 
	
	Combining now what we have proven we see that the set $\{t:(t,a_1,a_2)\in B\}$ is of the form $[t_{(a_1,a_2)},\infty)$.
	
	Secondly, we prove that $B$ is closed.\\\textbf{\underline{Compact Sections.}} \\ 
	Let $(x,a_1,a_2)\in B^{\complement}$. We will show that $B^{\complement}$ is an open set by constructing an open box $S$ such that $(x,a_1,a_2)\in S\subset \mathbb{R}^3\setminus B$. The case where $(a_1,a_2)\not \in \left(\mathbb{R}\setminus\mathbb{Q}\right)^2$ is the hardest one and the one we will prove here.

	Without loss of generality suppose that $a_1\in \mathbb{Q}$. Consider the minimal $k\in\mathbb{N}$ such that \[(a_1,a_2)\not \in \bigcup_{i,j\in\mathbb{N}} I_{n_0\dots n_{k-1}i}\times J_{m_0\dots m_{k-1}j},\] where $I_{n_0\dots n_k}$ and $J_{m_0\dots m_k}$ are the intervals we get using the method we described at the start of this section. There is a  $n_k$ such that for all $m\in\mathbb{Z}$, if we denote by $\zeta$ and $\zeta'$ the irrationals that correspond to the sequences $(n_0,\dots,n_k,m,\dots)$ and $(n_0,\dots,n_k+1,m,\dots) $ respectively then \begin{equation}\label{eq007}\zeta<a_1<\zeta'.\end{equation}
	We can now choose $N$ large enough so that no point $(x,\gamma_1,\gamma_2)$ belongs to $B$, where $\gamma_1$  satisfies \begin{equation}\label{eq15}(n_0,\dots,n_k,N,\dots)<\gamma_1<(n_0,\dots,n_k+1,-N,\dots),\end{equation}  $\gamma_2\in\mathbb{R}$ and we have abused the notation in the obvious way. The reason we can do this is because  for such $(\gamma_1,\gamma_2)$ we have that $R_{k+1}(x,\gamma_1,\gamma_2)=\emptyset$. Notice that by \eqref{eq007} we have that $a_1$ satisfies \eqref{eq15}.  It follows that there is a $\delta>0$ such that for all $y$ with $|y-x|<\delta$ we have that $R_{k+1}(y,\gamma_1,\gamma_2)=\emptyset$, for all $\gamma_1$ satisfying equations (\ref{eq15}) and $\gamma_2\in\mathbb{R}$.  
	
		Finally we prove density.
	\\\textbf{\underline{Density.}} \\We want to show that 
	\[
	A:=\{(a_1,a_2):t_{a_1,a_2}<\infty\}
	\]
	is dense in $\left(\mathbb{R}\setminus\mathbb{Q}\right)^2$. It is enough to show that the corresponding sequences in $\prod_{i=0}^{\infty}\mathbb{Z}^2$ of such points are dense in the space of all sequences. We know that the periodic sequences, meaning sequences with $(n_{k,1},n_{k,2})=(n_{k+N,1},n_{k+N,2})$, for some $N\in\mathbb{N}$ and all $k\in\mathbb{N}$, are dense in the space of all sequences. Hence, it is enough to prove that periodic sequences correspond to points in $A$. Indeed, for any periodic sequence it not hard to see that when $x$ is large enough the sequence $x_k$ is increasing and $R_k(x,a_1,a_2)\not=\emptyset$ for $k\leq N+1$ and as a result, since the sequence $(n_{k,1},n_{k,2})$ is periodic, for all $k\in\mathbb{N}$. This implies that $t_{(a_1,a_2)}<\infty$ for that $(a_1,a_2)$.
	
	Suppose now that $(x,a_1,a_2)\in B$ and $(a_1,a_2)$ has a corresponding sequence $\{(n_{k,1},n_{k,2})\}_{k\in\mathbb{N}}$. Choose $y_i$ with $x<y_i<x+\frac{1}{i}$, for all $i\in\mathbb{N}$. Then for all $i$ the inner radius of the shell $\mathcal{Z}_\lambda(R_{k-1}(y_i,a_1,a_2))$ will be much larger than the inner radius of $\mathcal{Z}_\lambda(R_{k-1}(x,a_1,a_2))$, for $k$ large enough. It follows that $(y_i,a_{1},a_{i,2})\in B$ where $(a_{1},a_{i,2})$ is chosen so that it has a corresponding sequence $$(n_{0,1},n_{0,2}),\dots (n_{i,1}, n_{i,2}-1),(n_{i+1,1}, n_{i+1,2}),(n_{i+2,1}, n_{i+2,2}),\dots$$ and thus $a_{i,2}$ is an increasing sequence. Hence,  $y_i\to x$ while $a_{i,2}\uparrow a_2$ as $i\to \infty$. Note that closedness of $B$ implies now that $t_{(a_1,a_{n,2})}\to t_{(a_1,a_2)}$.
	
	Similarly we will have that $(y_i,a_1,b_{i,2})$, $(y_i,c_{i,1},a_{2})$, $(y_i,d_{i,1},a_{2})$$\in B$, where $(a_{1},b_{i,2})$, $(c_{i,1},a_{2})$, $(d_{i,1},a_{2})$ have  corresponding sequences \[(n_{0,1},n_{0,2}),\dots (n_{i,1}, n_{i,2}+1), (n_{i+1,1}, n_{i+1,2}),(n_{i+2,1}, n_{i+2,2}),\dots,\]\[(n_{0,1},n_{0,2}),\dots (n_{i,1}-1, n_{i,2}), (n_{i+1,1}, n_{i+1,2}),(n_{i+2,1}, n_{i+2,2}),\dots,\]\[(n_{0,1},n_{0,2}),\dots (n_{i,1}+1, n_{i,2}), (n_{i+1,1}, n_{i+1,2}),(n_{i+2,1}, n_{i+2,2}),\dots\] respectively and $b_{i,2}$ is decreasing, $c_{i,1}$ is increasing and $d_{i,1}$ is decreasing.

\end{proof}

Next comes the construction of a suitable homemorphism $\varphi$ from $B$ to the Julia set $\mathcal{J}(\mathcal{Z}_\lambda)$.

\subsection{Proving that $B$ is homeomorphic to $\mathcal{J}(\mathcal{Z}_\lambda)$}
For each $(x,a_1,a_2)\in B$ we define \[\varphi(x,a_1,a_2):=\bigcap_{k=0}^{\infty}B_k.\]
Note that $B_k\subset\dots\subset B_0$, for all $k\in\mathbb{N}$ and also that $\diam(B_k)\to 0$ since all boxes $R_k(x,a_1,a_2)$  are inside the half space  $\{(x_1,x_2,x_3):x_3\geq p_\lambda\}$ on which, by Lemma \ref{inverseexpansion} and by Lemma \ref{pl-lemma},  the Zorich map is expanding. Thus \[\diam(B_k)=\sup_{z,y\in R_k(x,a_1,a_2)} \left|\Lambda^k(z)-\Lambda^k(y)\right|\leq \alpha^k\diam(R_k(x,a_1,a_2)).\]

\begin{lemma}
	$\varphi$ is an injective continuous map from $B$ into the Julia set $\mathcal{J}(\mathcal{Z}_\lambda)$.
\end{lemma}
\begin{proof}
	Let us first show that $\varphi$ is injective. Let $(x,a_1,a_2), (y,b_1,b_2)$ be two different points of $B$ and let $\{(n_{k,1},n_{k,2})\}_{k\in\mathbb{N}}$ and $\{(m_{k,1},m_{k,2})\}_{k\in\mathbb{N}}$ their corresponding sequences. Now either $x\not=y$ or $x=y$ and $(a_1,a_2)\not=(b_1,b_2)$. In the first case we may assume $x<y$ and thus $x_k<y_k$ for all $k\in\mathbb{N}$. We can easily see by the mapping properties of the Zorich map that in fact $y_k-x_k\to\infty$ for $k\to \infty$. Thus for large enough $k$ we can find cubes $R_k(x,a_1,a_2)\cap R_k(y,b_1,b_2)=\emptyset$, which implies that $\varphi(x,a_1,a_2)\not=\varphi(y,b_1,b_2)$. In the second case, we have that $(n_{k,1},n_{k,2})\not=(m_{k,1},m_{k,2})$ for some $k\in\mathbb{N}$ and thus $R_k(x,a_1,a_2)\cap R_k(y,b_1,b_2)=\emptyset$.
	
	Now for the continuity of $\varphi$ let $(x,a_1,a_2)$ and $(y,b_1,b_2)$ be two points in $B$ that are "close" meaning that their corresponding sequences satisfy $n_{k,i}=m_{k,i}$, $i=1,2$, for $k=0,\dots,N$ where $N\in \mathbb{N}$ and $|x-y|<\delta$ for some small $\delta>0$. If $\delta$ is small enough this implies that \[R_k(x,a_1,a_2)\cap R_k(y,b_1,b_2)\not=\emptyset, \hspace{1mm}\text{for all}\hspace{1mm}k\leq N.\] Hence, $B_k(x)\cap B_k(y)\not=\emptyset$ for all $k\leq N$. This implies that \[|\varphi(x,a_1,a_2)-\varphi(y,b_1,b_2)|\leq \diam B_N(x)+\diam B_N(y)\] and because those diameters tend to zero we get that when $N$ is large enough $\varphi(x,a_1,a_2)$ and $\varphi(y,b_1,b_2)$ are close.
	
	Finally, we want to show that $\varphi(x,a_1,a_2)\in\mathcal{J}(\mathcal{Z}_\lambda)$. If $\{(n_{k,1},n_{k,2})\}_{k\in\mathbb{N}}$ is the corresponding sequence of $(a_1,a_2)$ then, by construction and \eqref{juliasquare}, we will have that \[R_k(x,a_1,a_2)\cap \mathcal{J}(\mathcal{Z}_\lambda)\not=\emptyset, \hspace{1mm}\text{for all}\hspace{1mm}k\in\mathbb{N}.\] This implies that $d\left(\varphi(x,a_1,a_2),\mathcal{J}(\mathcal{Z}_\lambda)\right)=0$ and because $\mathcal{J}(\mathcal{Z}_\lambda)$ is closed  we have that   $\varphi(x,a_1,a_2)\in\mathcal{J}(\mathcal{Z}_\lambda)$.
\end{proof}
 
 Next we need to find the inverse of $\varphi$ and show that it is continuous. Let us first define this function and then show that indeed it is the inverse of $\varphi$.
 
 Let $w=(w_1,w_2,w_3)\in\mathcal{J}(\mathcal{Z}_\lambda)$. Remember now that with each $w$ in the Julia set we can associate its itinerary $\Delta(w)=n_0n_1n_2\dots$ and $\mathcal{Z}^k_\lambda(w)\in P(n_k)\times \mathbb{R}$. For any $k\in\mathbb{N}$ define  the boxes $T_j(k)$, $j\in \mathbb{N}$ as follows. First  we set $T_k(k)=S(u,n_{k,1},n_{k,2})$, where $u$ is minimal with respect to the properties $u\geq p_\lambda$ and $\mathcal{Z}^k_\lambda(w)\in T_k(k)$. We now define $T_j(k)$ for $j=0,\dots, k-1$ inductively as follows.  Suppose that $T_j(k)$ has been defined, let $T_{j-1}(k)=S(v,n_{j-1,1},n_{j-1,2})$, where $v$ is maximal with respect to the property \begin{equation}\label{eq100}
 T_j(k)\subset \mathcal{Z}_\lambda(T_{j-1}(k)).
 \end{equation}
Note that a $v$ for which \eqref{eq100} holds exists since $u\geq p_\lambda>1$. Moreover, \eqref{eq100} thanks to the continuity of $\mathcal{Z}_\lambda$, implies that the box $T_j(k)$ hits the inner radius of the half-shell $\mathcal{Z}_\lambda(T_{j-1}(k))$ in exactly one point. Also it implies that the lowest (in terms of $x_3$ coordinate) side of the box $T_{j-1}(k)$ has a third coordinate at least $p_\lambda$ for all $j$. This can be shown by first noting that it is true for $T_{k-1}(k)$ and then we can use induction on $j$ to prove it for all $j$. To see why it is true for $T_{k-1}(k)$ note that if that was not the case and $p_\lambda=J-1$ (see Lemma \ref{pl-lemma}) then $T_{k-1}(k)$ would not contain any points in the Julia set and since $T_k(k)$ does (by Lemma \ref{pl-lemma}) we get a contradiction by \eqref{eq100} and the invariance of the Julia set. If on the other hand $p_\lambda=M$ then the lowest side of the box $T_{k-1}(k)$ would be below $M$ and since $\mathcal{Z}_\lambda(H_{\leq M})\subset H_{<M}$ we get a contradiction by the maximality condition on \eqref{eq100}. Thus in any case the lowest (in terms of $x_3$ coordinate) side of the box $T_{k-1}(k)$ has a third coordinate at least $p_\lambda$.

 We now define $z_k$ by the condition $T_0(k)=S(z_k,n_{0,1},n_{0,2})$ and note that $w\in T_0(k)$. This implies that \begin{equation}\label{eq123}w_3-1\leq z_0\leq z_k \leq w_3,\end{equation}  for all $k\in\mathbb{N}$. We can also prove that  \begin{equation}\label{eq1234}
z_k\leq z_{k+1}.
\end{equation} This follows by the fact that $T_k(k+1)$ is higher than $T_k(k)$ (in terms of $x_3$ coordinate) and thus by the inductive construction  \eqref{eq1234} holds.
 
 Finally, we set $z_{\infty}:=\lim_{k\to\infty}z_k$ and  define $\psi$ by \[\psi(w)=\left(z_{\infty},a_1,a_2\right),\]
 where $(a_1,a_2)$ is the pair of irrationals associated with the pair of sequences $\Delta(w)$.  The next lemma concludes the proof of Theorem \ref{brush}.
 
\begin{lemma}
	The map $\psi$ is the inverse of $\varphi$ and $\varphi$ is a homeomorphism. Moreover, $\varphi$ extends to a homeomorphism between $B\cup\{\infty\}$ and $\mathcal{J}(\mathcal{Z}_\lambda)\cup\{\infty\}$.
\end{lemma}
\begin{proof}
	We will first show that \begin{equation}\label{eq109}\varphi\circ\psi=\text{id}_{\mathcal{J}(\mathcal{Z}_\lambda)}.\end{equation}
	Let $w\in\mathcal{J}(\mathcal{Z}_\lambda)$ and we continue using notation as above.
	We define $T_0(\infty)=S(z_\infty,n_{0,1},n_{0,2})$ and note that $T_0(\infty)=\lim_{k\to\infty}T_0(k)$ in the Hausdorff metric. Define also $T_j(\infty):=\lim_{k\to\infty}T_j(k)$, where the limit is again in the Hausdorff metric. Notice that by construction the box $T_j(\infty)$ hits the inner radius of the half-shell $\mathcal{Z}_\lambda(T_{j-1}(\infty))$ in exactly one point and that point has a third coordinate at least $p_\lambda$.  Applying now the construction of $\varphi$ to the point $\psi(w)=\left(z_\infty, a_1,a_2\right)$,  where $(a_1,a_2)$ is the pair of irrationals associated with the pair of sequences $\Delta(w)$ and using the notation of the construction of $B$ we get that $x_0=z_\infty$ and \[R_k(x,a_1,a_2)=T_k(\infty),\] for all $k\in\mathbb{N}$. Hence, $\psi(w)\in B$ and \[\varphi(\psi(w))=\varphi(z_\infty,a_1,a_2)=\bigcap_{k\in\mathbb{N}}\mathcal{Z}^{-k}_\lambda(T_k(\infty)),\]
	where ${Z}^{-k}_\lambda$ is the composition of inverse branches of $\mathcal{Z}_\lambda$ to appropriate square beams (see construction of $\varphi$).
	Now note that by equation (\ref{eq123}) taking limits we have that $w\in T_0(\infty)$. We can prove similar equations with (\ref{eq123}), (\ref{eq1234}) for all boxes $T_j(k)$. Hence, we will also have that $\mathcal{Z}^k_\lambda(w)\in T_k(\infty)$. This implies that $$w\in \bigcap_{k\in\mathbb{N}}\mathcal{Z}^{-k}_\lambda(T_k(\infty))$$ and since that intersection contains only one point we get that $\varphi(\psi(w))=w$.
	
	Equation (\ref{eq109}) implies now that $\varphi$ is onto the Julia set and thus a bijection. Since we have already proven that $\varphi$ is continuous we can now extend $\varphi$ to a continuous map $\hat{\varphi}$, from $B\cup\{\infty\}$ to the one point compactification  $\mathcal{J}(\mathcal{Z}_\lambda)\cup\{\infty\}$ of $\mathcal{J}(\mathcal{Z}_\lambda)$, with the spherical metric by setting $\hat{\varphi}(\infty)=\infty$, since $\varphi(b)\to\infty$ as $b\to\infty$, $b\in B$. We can now use a well known lemma (see for example \cite[Theorem 4.17]{Rudin1976}) and immediately conclude that in fact $\hat{\varphi}$ is a homeomorphism (with the spherical metric) and $\hat{\psi}$, defined as the extension of $\psi$ on $\mathcal{J}(\mathcal{Z}_\lambda)\cup\{\infty\}$ by setting $\hat{\psi}(\infty)=\infty$,  is its inverse. This implies that $\psi$ is continuous with the euclidean metric.
\end{proof}

\section{Proving Theorem \ref{fan}}
First we will show that the model of the Julia set, namely the 3-d straight brush, is a Lelek fan.

\begin{lemma}\label{brushlelek}
	The one point compactification $B\cup\{\infty\}$ of a  3-d straight brush $B$ is a Lelek fan with top at $\infty$.
\end{lemma}
\begin{proof}
	Any subcontinuum of $B\cup\{\infty\}$ is a collection of segments of straight lines together with $\infty$ or a segment of one such line. Obviously the intersection of any two such subcontinua is again of that form and thus connected. Properties (ii) and (iii) of the definition of a fan are obviously satisfied for $B\cup\{\infty\}.$ Smoothness is also quite easy to prove since arcs $[y_n,\infty]$ in $B\cup\{\infty\}$ are straight lines and thus converge to a straight line $[y,\infty]$ in the Hausdorff metric, when $y_n\to y$.
	
	Finally, we prove that the set of endpoints $\mathcal{E}(B)$ is dense in $B$. Choose a point $(x,b_1,b_2)\in B$. We will  show that there are hairs with \begin{equation}\label{eq001}(t_{b_{1,n},b_{2,n}},b_{1,n},b_{2,n})\to(x,b_1,b_2),\end{equation} as $n\to \infty$. In other words their endpoints converge to $(x,b_1,b_2)$ and this obviously gives us the density of endpoints.

	 Consider the length function $\mathcal{L}:\left(\mathbb{R}\setminus\mathbb{Q}\right)^2\to [0,\infty)$ which measures the spherical length of the hair at $(a_1,a_2)$. It is easy to see that thanks to the closedness of $B$, $\mathcal{L}$ is upper semi-continuous. Let $$M:=\sup\{\mathcal{L}(a_1,a_2):(a_1,a_2)\in\left(\mathbb{R}\setminus\mathbb{Q}\right)^2\} . $$ Take now any $c\in(0,M)$ and consider the set \[V:=\{(a_1,a_2)\in\left(\mathbb{R}\setminus\mathbb{Q}\right)^2:\mathcal{L}(a_1,a_2)>c\}.\]
	 
	 	Fix now a point $(b_1,b_2)$ in $V$ and consider the set $$W:=\{(b_1,a_2)\in\overline{V}: |a_2|\leq t,  \mathcal{L}(b_1,a_2)> c\}, $$ where $t>0$ a constant.
	 
	  From the upper semi-continuity of $\mathcal{L}$ we conclude that the set \[\{(a_1,a_2)\in\left(\mathbb{R}\setminus\mathbb{Q}\right)^2:\mathcal{L}(a_1,a_2)\geq c\}\] is closed and thus contains $\overline{W}$. 
	
 The set $\overline{W}$ is compact and lies on a line. By property (ii) of the definition of a 3-d straight brush now we will have that it is also perfect. Moreover, it is a totally disconnected set, since it is a subset of a totally disconnected set. Hence, by the characterization of the Cantor set (see for example \cite[Theorem 6.17]{Cannon2017} or the more general \cite[Theorem 7.4]{Kechris1995}) we get that $\overline{W}$ is a Cantor set (i.e. homeomorphic to the standard ternary Cantor set $\mathcal{C}$). In fact, the sets $\overline{W}$ and $\mathcal{C}$ are ambiently homeomorphic (meaning here that the homeomorphism extends to the line on which $\overline{W}$ lies). This implies that the obvious order in the sets is either preserved or reversed under the homeomorphism. This in turn implies, that a point $(b_1,a_2)\in \overline{W}$ for which $\mathcal{L}(b_1,a_2)>c$ cannot get mapped to an endpoint of the Cantor set since, thanks to property (ii) of a 3-d straight brush, it can be approximated from both sides by other points  of $\overline{W}$ (unless of course $a_2=\pm t$ but those are just two points). Hence, only points with $\mathcal{L}(b_1,a_2)=c$ get mapped to endpoints and since endpoints are dense in the Cantor set so are  points with $\mathcal{L}(b_1,a_2)=c$ dense in $\overline{W}$.
	
	Since this argument works for any $(b_1,b_2)\in V$  we have that points with $\mathcal{L}(a_1,a_2)=c$ are dense in $\overline{V}$.

	Take now any point $(x,b_1,b_2)\in B$, which is not an endpoint, and let $c$ be the spherical length of the hair segment $\{(t,b_1,b_2):t\geq x\}$. This  implies that  $\mathcal{L}(b_1,b_2)>c.$  As we have shown then there is a sequence of hairs $(t_{b_{1,n},b_{2,n}},b_{1,n},b_{2,n})$ of length $c$ for which  $(b_{1,n},b_{2,n})\to (b_1,b_2)$ and thus the endpoints of the hairs $(t_{b_{1,n},b_{2,n}},b_{1,n},b_{2,n})$ converge to $(x,b_1,b_2)$ as we wanted.
\end{proof}

\begin{proof}[Proof of Theorem \ref{fan}]
From Theorem \ref{brush} we know that there  a homeomorphism $$\hat{\varphi}:B\cup\{\infty\}\to\mathcal{J}(\mathcal{Z}_\lambda)\cup\{\infty\}.$$ By Lemma \ref{brushlelek} we know  that $B\cup\{\infty\}$ is a Lelek fan. It is quite easy to see that all the properties of a Lelek fan are preserved under a homeomorphism and thus $\mathcal{J}(\mathcal{Z}_\lambda)$ will also be a Lelek fan with top at $\infty$. 
\end{proof}
\section{Hairy  squares and hairy surfaces}
\label{hairysurfaces}

In this section we generalize the notion of a hairy arc that Aarts and Oversteegen first introduced in \cite{Aarts1993}. Our exposition closely follows theirs but as we shall see some things are different in three dimensions.

First we will introduce the notion of a \textit{straight one-sided hairy square} (abbreviated soshs) which will be a generalization of \textit{straight one-sided hairy arcs} (abbreviated sosha) on the plane (see \cite{Aarts1993} for that definition). Let $I=[0,1]$.
\begin{definition}[Straight one-sided hairy square]
	A straight one-sided hairy square $X$ is a compact subset of $I^3$ satisfying the following properties. There is a function $\ell:I^2\to I$, called the length function, such that \begin{enumerate}[label=(\roman*)]
		\item For all $(x,y,z)\in I^3$ we have $(x,y,z)\in X$ if and only if $0\leq z\leq\ell(x,y)$.
		\item The sets $$\{(x,y)\in I^2:\ell(x,y)>0\}, \hspace{1mm}\{x\in I:\ell(x,y)=0, \forall y\in I\}\hspace{1mm}\text{and}\hspace{1mm} \{y\in I:\ell(x,y)=0, \forall x\in I\}$$ are  dense in $I^2$, $I$ and $I$ respectively and $\ell(0,t)=\ell(1,t)=\ell(t,0)=\ell(t,1)=0$ for all $t\in[0,1]$.
		\item For each $(x,y)\in I^2$ with $\ell(x,y)>0$ there exist sequences  $a_{n}$, $b_{n}$,  $c_{n}$,  $d_{n}$  such that $a_{n}\uparrow y$, $ b_{n}\downarrow y$, $c_{n}\uparrow x$, $d_{n}\downarrow x$. Moreover it is true that $\ell\left(x,a_{n}\right)\to \ell(x,y)$ and similarly for the other sequences. 
	\end{enumerate}
\end{definition}
  For each $(x,y)\in I^2$, the set $\{(x,y,z):0\leq z\leq\ell(x,y) \}$ will be called the \textit{hair} at $(x,y)$ while the set $I^2\times \{0\}$ will be called the \textit{base}.

  The usefulness of the above object lies in the fact that when we suitably embed a 3-d straight brush to $I^3$ with the usual topology and then compactify that embedding we get a soshs.
  
  Indeed, let $\mathcal{H}:\mathbb{R}^3\to I^3$ be defined as 
  
  \begin{equation}\label{embe}\mathcal{H}(x,y,z)=\left(\frac{\arctan y}{\pi}+\frac{1}{2},\frac{\arctan z}{\pi}+\frac{1}{2},\frac{\mathcal{L}\left([x,\infty)\times \{(y,z)\}\right)}{\pi}\right)\end{equation}
  where $\mathcal{L}$ is once again the spherical length of the half line $[x,\infty)\times \{(y,z)\}$. We note here that we are viewing $\overline{\mathbb{R}^3}$ as the sphere of centre $(0,0,1/2)$ and radius $1/2$ so the spherical length of any straight line is less or equal than $\pi$. 
  
  It is easy to see now that $\mathcal{H}$ is an embedding of $\mathbb{R}^3$ to $I^3$
  and that if  $B$ is a 3-d straight brush then the compactification of $\mathcal{H}(B)$ in $I^3$, with usual topology is a soshs.

  After they defined straight one-sided hairy arcs, Aarts and Oversteegen went on to define the notions of a \textit{hairy arc} and a \textit{one-sided hairy arc} which are a homeomorphic image of a sosha and a homeomorphic image of a sosha on the plane with all of the hairs attached on the same side respectively. The importance of those objects becomes apparent when we notice that  compactified versions of Julia sets of many entire transcendental functions are one-sided hairy arcs. Initially, Aarts and Oversteegen showed this for some functions in the exponential family and some functions in the sine and cosine families. That was extended to much larger classes of transcendental entire maps by Baranski, Jarque and Rempe in \cite{Baranski2012}. 
  
  Generalizing  in $\mathbb{R}^3$ we define the notions of a \textit{hairy surface} and a \textit{one-sided hairy surface}.
  
  \begin{definition}[Hairy Surface]
  	A hairy surface is any homeomorphic image of a soshs. The base of the hairy surface is the image of $I^2\times\{0\}$ under that homeomorphism and the hairs are the images of the hairs of the soshs. A one-sided hairy surface  is an embedding $\varphi$ of a hairy surface $X$, with base $D$, in $\mathbb{R}^3$ such that all hairs are attached to the same side of the base $\varphi(D)$. 
  \end{definition}
  The notion of same sidedness is intuitively clear. Rigorously we can define it as follows. There is a surface $\mathcal{S}$ which contains $\varphi(D)$ and all of the hairs are contained in the same bounded complementary component of $\mathcal{S}$.

  Next we prove that when we suitably compactify the Julia sets of some Zorich maps we do get hairy surfaces. 
 
\begin{proof}[Proof of Theorem \ref{compactify}]
	We know from section \ref{section3} that there is a homeomorphism $\psi:\mathcal{J}(\mathcal{Z}_\lambda)\to B$, where $B$ is a 3-d straight brush. Consider now the embedding $f:\mathcal{J}(\mathcal{Z}_\lambda)\to I^3$, with $f=\mathcal{H} \circ \psi$, where $\mathcal{H}:B \to I^3$ was defined in \eqref{embe}. The compactification we are looking for is the one induced by the embedding $f$ (see for example \cite[Chapter 5.3]{Munkres1974}). Also, by construction, $f$ extends to a homeomorphism $\widetilde{f}$ between the compactification of $\mathcal{J}(\mathcal{Z}_\lambda)$, $\widetilde{\mathcal{J}(\mathcal{Z}_\lambda)}$ and $\overline{\mathcal{H}(B)}$ the closure of $\mathcal{H}(B)$ in the usual topology of $I^3$. Hence, $\widetilde{\mathcal{J}(\mathcal{Z}_\lambda)}$ is a  hairy surface.
\end{proof}

\section{Wild one sided hairy surfaces}
When studying embeddings of a subset $X$ of $\mathbb{R}^d$  in $\mathbb{R}^d$ an important notion is that of \textit{tameness} of $X$.  In other words, whether or not there is a homeomorphism $H:\mathbb{R}^d\to\mathbb{R}^d$ sending $X$ to $h(X)$, where $h:X\to   \mathbb{R}^d$ is a homeomorphism onto its image. We also say that $X$ and $h(X)$ are then ambiently homeomorphic.

Aarts and Oversteegen in \cite{Aarts1993} showed that  all one-sided hairy arcs are tame. In other words any embedding $\varphi:X\to\mathbb{R}^2$ of a sosha $X$  for which the hairs of $\varphi(X)$ are all on the same side extends to a homeomorphism of the whole plane. This is reminiscent of the tameness of the Cantor set and the arc in the plane. Once we move however to higher dimensions the situation is different. It is well known that in $\mathbb{R}^3$, for example, there are wild arcs, wild Cantor sets and wild spheres meaning sets that are homeomorphic images of $[0,1]$, the standard ternary Cantor set and the unit sphere respectively and yet they are not ambiently homeomorphic to those sets. Classical examples of such sets are the Wild arc, Antoine's necklace and Alexander's horned sphere (see for example \cite{Moise1977} and references therein).

The situation is similar for soshs' and one sided hairy surfaces as Theorem \ref{wild} shows.

In order to prove our theorem we will first construct a soshs $S$ and then a homeomorphism which takes that soshs to a wild one sided hairy surface. 

\subsection{Constructing a soshs} 
For convenience we will construct our soshs as a subset of $I^2\times [0,3/2]$ instead of $I^3$.

Our soshs will be constructed as an intersection of subsets $T_n$ of $I^2\times [0,3/2]$ each of which comprises of a square base $I^2$ and a set of cuboids $R(n,i,j)$. The cuboids will have a square base on $I^2$ and will have their sides parallel with the axis $x$, $y$ and $z$. The construction proceeds inductively. Set $T_0=T_1=I^2\times [0,3/2]$ and partition $I^2$ in 9 equal squares $Q(2,i,j)$, $i,j=1,2,3$ which will be the base of our cuboids $R(2,i,j)$.  On each of those squares we erect a cuboid of height as in Figure \ref{Squares}. The bottom left corner square is $Q(2,1,1)$ and the top right one is $Q(2,3,3)$.
\begin{figure}[h]
	\centering
	\setlength\tabcolsep{0pt}
	\begin{tabular}{|@{\rule[-0.4cm]{0pt}{1.2cm}}*{3}{M{1.2cm} |}}
		\hline
		1/2 & 1/2 & 1/2\\
		\hline
		1/2 & 3/2 & 1/2 \\
		\hline
		1/2 & 1/2 & 1/2 \\
		\hline
	\end{tabular}
	\caption{The squares $Q(2,i,j)$ and the heights of the corresponding cuboids}
	\label{Squares}
\end{figure}

We now name $T_2=\bigcup_{i,j\in\{1,2,3\}}R(2,i,j)$. Suppose now that we have constructed $T_n$ and it is a finite union of cuboids $R(n,i,j)$ with bases $Q(n,i,j)$. Take each square  base $Q(n,i,j)$ and partition it in $(2n+1)^2$ equal squares, which we will now call  $Q(n,i,j,k,l)$. After we define $T_{n+1}$ we relabel those squares as $Q(n+1,i,j).$  Notice that we have used the same indices $i,j$ although they take different values for different values of $n$. Also, let us denote by $h(n,i,j)$ the height of the cuboid $R(n,i,j)$.

We now erect cuboids $R(n,i,j,k,l)$ with bases $Q(n,i,j,k,l)$ with heights \[h(n,i,j,k,l)=\begin{cases}
\frac{1}{n+1}h(n,i,j), &\hspace{2mm}\text{if}\hspace{2mm} l=1 \hspace{2mm}\text{or}\hspace{2mm} l=2n+1\hspace{2mm}\text{or}\hspace{2mm} k=1\hspace{2mm}\text{or}\hspace{2mm} k=2n+1\\\\
h(n,i,j),& \hspace{2mm}\text{if}\hspace{2mm} Q(n,i,j,k,l)\hspace{2mm}\text{is the central square of}\hspace{2mm} Q(n,i,j)\\\\
\frac{n}{n+1}h(n,i,j), &\hspace{2mm}\text{otherwise}.

\end{cases}\]
Then we define $T_{n+1}=\bigcup_{i,j,k,l}R(n,i,j,k,l).$\\

\begin{figure}[H]
	\centering
	\setlength\tabcolsep{0pt}
	\begin{tabular}{|@{\rule[-0.4cm]{0pt}{1cm}}*{5}{M{1cm} |}}
		\hline
		1/2 & 1/2 & 1/2 & 1/2 & 1/2 \\
		\hline
		1/2 & 1 & 1 & 1 & 1/2 \\
		\hline
		1/2 & 1 &  3/2 & 1 & 1/2 \\
		\hline
		1/2 & 1 & 1 & 1 & 1/2 \\
		\hline
		1/2 & 1/2 & 1/2 & 1/2 & 1/2 \\
		\hline
	\end{tabular}
	\caption{The squares $Q(2,2,2,k,l)$ and the heights $h(2,2,2,k,l)$ of the corresponding cuboids}
	\label{Squares2}
\end{figure}

Our soshs will then be defined as $S=\bigcap_{n=0}^{\infty}T_n$. Since every $T_n$ is a continuum and $T_{n+1}\subset T_n$, $S$ will also be a continuum and it is easy to see that it will satisfy all of the properties of a soshs. We will only prove here that property (iii) of the definition of a soshs is satisfied. 

Indeed, for any hair $\delta:[0,1]\to S$ of $S$ at the base point $(x,y)$  it will be true that there is a sequence of cuboids $R(n,i_n,j_n)$ with $R(k+1,i_{k+1},j_{k+1})\subset R(k,i_k,j_k)$, for all $k\in\mathbb{N}$ and $$\bigcap_{n}R(n,i_n,j_n)=\delta\left([0,1]\right).$$

It is true now that there is a subsequence $n_k$ such that $R(n_k,i_{n_k},j_{n_k})$ is the central square of $R(n_k-1,i_{n_k-1},j_{n_k-1})$ since otherwise  there would not be a hair at $(x,y)$.

We now find a sequence of hairs $\delta_k$, $k=1,2,\dots$ with base points $(x_k,y)$ such that $x_k\uparrow x$ and  $\ell(x_k,y)\to \ell(x,y)$. We choose the hair $\delta_k$ as the one defined by the sequence of cuboids \begin{align*}R(1,i_1,j_1),&\dots,R(n_k-1,i_{n_k-1},j_{n_k-1}), R(n_k,i_{n_k}-1,j_{n_k}),\\&R(n_k+1,m_{1,k},j_{k+1}), R(n_k+2,m_{2,k},j_{n_k+2}),\dots,\end{align*}
where $m_{i,k}$, $i\in\mathbb{N}$ is a sequence of integers such that the sequence of squares \[Q(n_k+1,m_{1,k},j_{k+1}), Q(n_k+2,m_{2,k},j_{n_k+2}),\dots\] is the sequence of squares \[Q(n_k+1,i_{n_k+1},j_{k+1}),  Q(n_k+2,i_{n_k+2},j_{n_k+2}),\dots\] translated to the left by the length of a side of a square at the $n_k$ level.

It is now easy to see that $\delta_k$ has its base point at $(x_k,y)$ with $x_k\uparrow x$. Also the hairs $\delta_k$ have length $$\ell(x_k,y)=\frac{n_k}{n_k+1}\ell (x,y)$$ so that $\ell(x_k,y)\to\ell(x,y)$ as $k\to\infty$.

Similarly we can construct the other sequences of points from property (iii) of the definition of a soshs.
\subsection{Proof of Theorem \ref{wild}}
Before we proceed with the proof of Theorem \ref{wild} let us first introduce some standard terminology taken from \cite{Moise1977}. 

Let $V=\{u_0, u_1, \dots, u_n\} $ be a set of $n+1$  points in $ \mathbb{R}^d$ which are affinely independent, meaning $u_1-u_0$, $u_2-u_0$, $\dots, u_n-u_0$ are linearly independent. Then the \textit{$n$-simplex} is defined as the convex hull of $V$. We will denote the $n$-simplex defined by those points by $\sigma^n$. The convex hull $\tau$ of a non empty subset $W$ of $V$ will be called a face of $\sigma^n$. A \textit{(Euclidean) complex} is a collection $\mathcal{K}$ of simplexes in $\mathbb{R}^d$ such that \begin{enumerate}
	\item $\mathcal{K}$ contains all faces of all elements of  $\mathcal{K}$.
	\item If $\sigma$, $\tau\in \mathcal{K}$ are simplexes and $\sigma\cap \tau\not= \emptyset$ then $\sigma\cap \tau$ is a face of both $\sigma$ and $\tau$.
	\item Every $\sigma$ in $\mathcal{K}$ lies in an open set $U$ which intersects only a finite number of members of $\mathcal{K}$. 
\end{enumerate}

If $\mathcal{K}$ is a complex then with $|\mathcal{K}|$ we denote the union of the elements of $\mathcal{K}$. Such a set is called a \textit{polyhedron}. Note that a polyhedron can be seen as a manifold with boundary with the subspace topology induced from the standard topology of $\mathbb{R}^d$.   An $n$-\textit{cell} is a space homeomorphic to an $n$-simplex. A \textit{polyhedral~$n$-cell} is a polyhedron homeomorphic to an $n$-simplex. 

Let $M$ be a manifold with boundary. Then by $\bd M$ and $\interior M$ we denote its boundary and its interior respectively.

\begin{lemma}\cite[Theorem 1, Chapter 19]{Moise1977}\label{moise}
	Let $A$ be a polyhedral $1$-cell in $\mathbb{R}^3$ with endpoints $P$ and $Q$. Then
	there is a polyhedral 3-cell $C$ such that (1) $ \interior A \subset\interior C$, (2) $P, Q \in
	\bd C$, and (3) there is a homeomorphism $\phi: C \to \sigma^2 \times [0,1]$, such that $A\mapsto R \times [0,1]$,
	for some $R \in \interior \sigma^2$.
\end{lemma}

If $A$ and $C$ satisfy the conditions of the above lemma, then we say that $A$ is \textit{unknotted} in $C$.

\begin{proof}[Proof of Theorem \ref{wild}]
	Consider the soshs $S$ we constructed in the previous subsection. We will construct a homeomorphism by defining it on the sets $R(n,i,j)$ for all $n$. Consider now cuboids $B_k$, $k\in\mathbb{N}$ with sides parallel to the $x$, $y$ and $z$ axis which are constructed as follows: 
\begin{itemize}
	\item $B_1$ has a square base $I^2\times \{0\}$ and height $h_1=1/2$.
	\item $B_2$ has as base $Q(2,2,2)\times \{1/2\}$ and height $h_2=\frac{1}{4}$.\\
\resizebox{0.05\hsize}{!}{$	\vdotswithin{ = }\notag \\$}

\item $B_n$ has as base $Q(n,i_n,j_n)\times \{\sum_{k=1}^{n-1}h_{k}\}$, where $Q(n,i_n,j_n) $ is the central square of\\ $Q(n-1,i_{n-1},j_{n-1})$ and height $h_{n}= \frac{1}{2^n}$.\\
\resizebox{0.05\hsize}{!}{$	\vdotswithin{ = }\notag \\$}
\end{itemize}

\begin{figure}
	\includegraphics[scale=0.25]{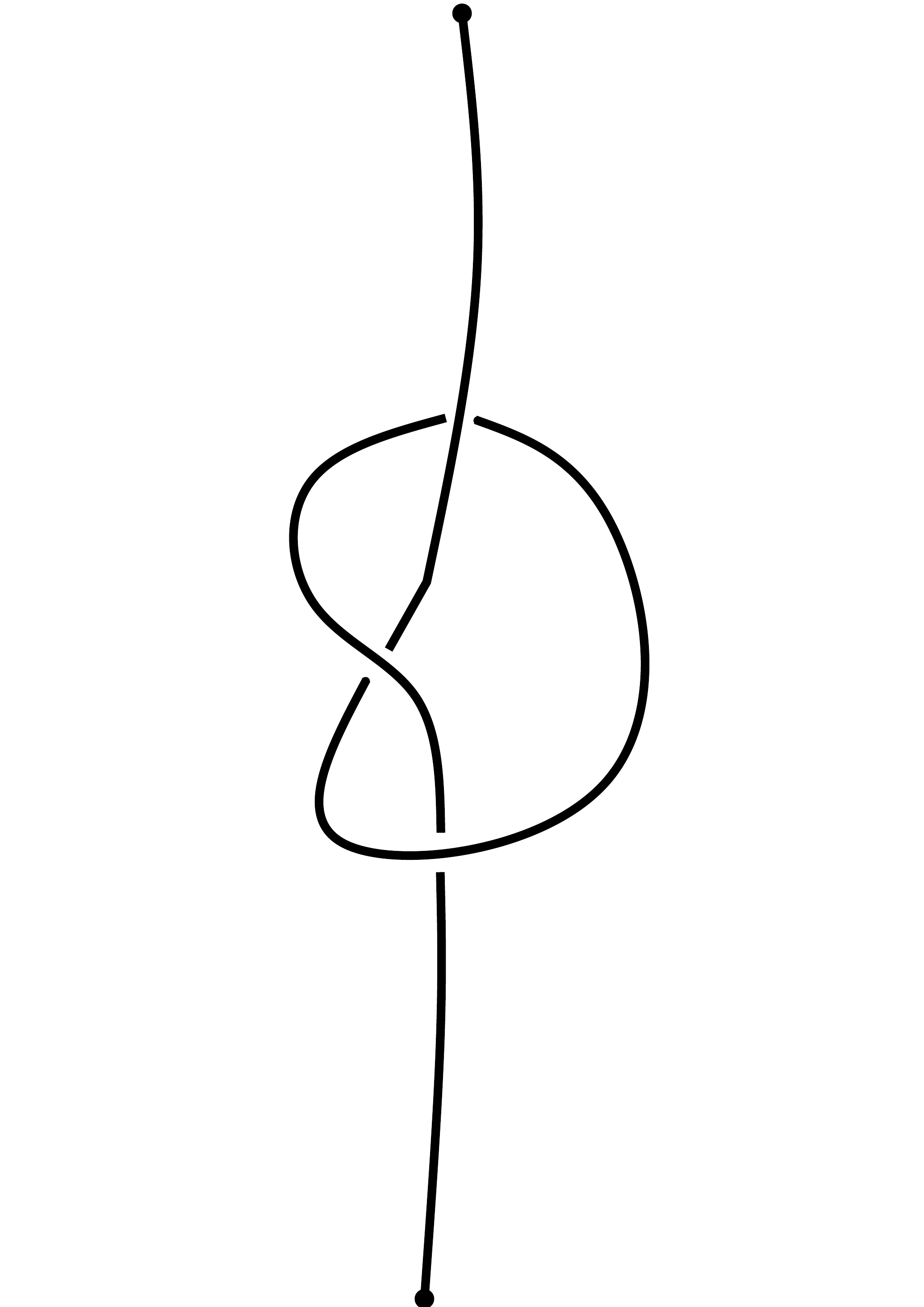}\hspace{3cm}
	\includegraphics[scale=0.25]{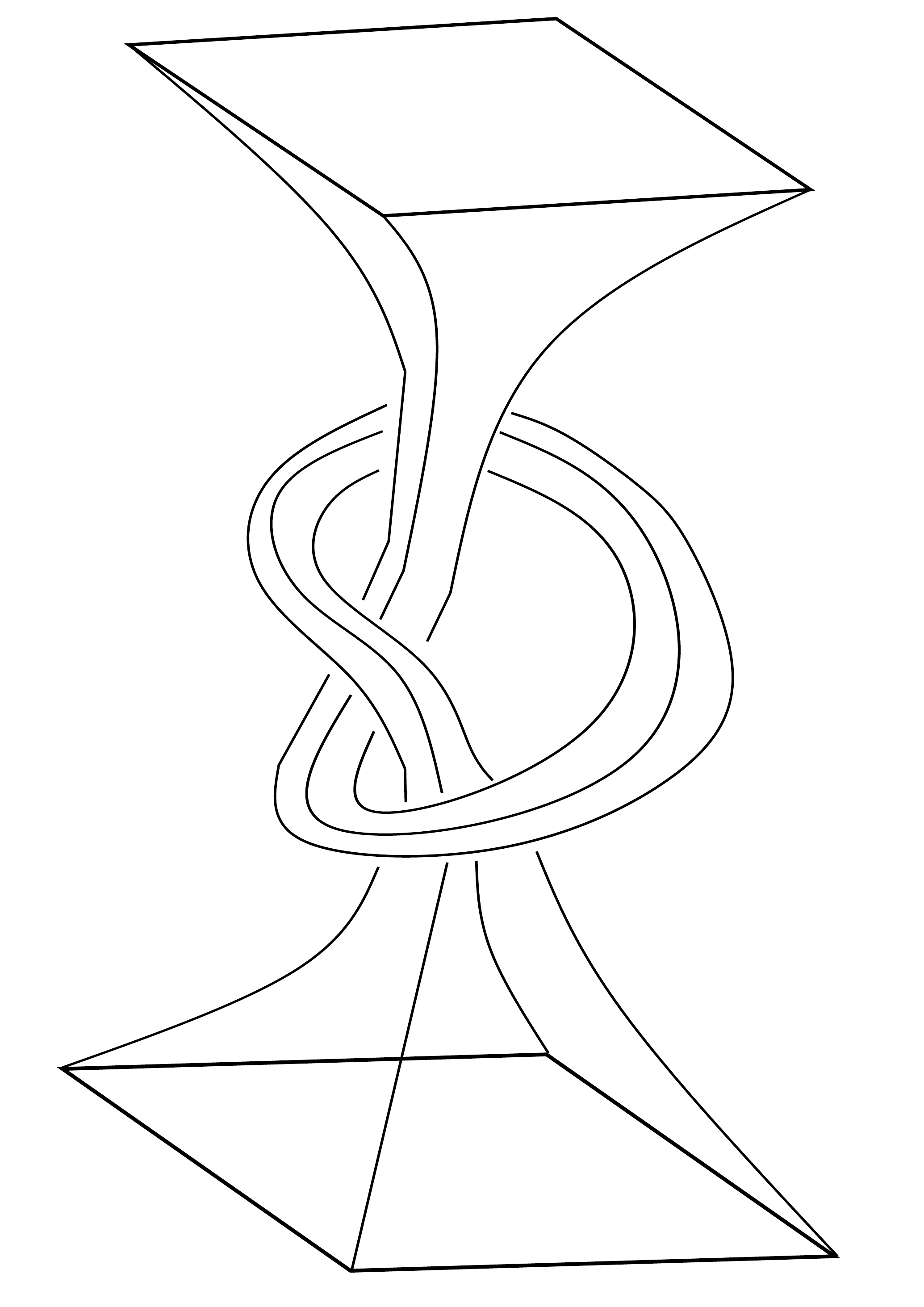}
	\caption{The polyhedral $1$-cells $A_n$ (left) and the polyhedral $3$-cells $C_n$ (right) used in the Proof of Theorem \ref{wild}.}
	\label{knots}
\end{figure}

Consider now the polyhedral $1$-cells $A_n$ inside $B_n$ that are overhand knotted (see Figure \ref{knots}) with their endpoints at $(\frac{1}{2},\frac{1}{2},h_{n-1})$ and $(\frac{1}{2},\frac{1}{2},h_{n})$. From Lemma \ref{moise} now, for each of those $1$-cells, we can find a polyhedral $3$-cell $C_n$, inside $B_n$, in which $A_n$ is unknotted and $\bd C_n$ contains the bottom and the top square bases of $B_n$.

 Lemma \ref{moise} also gives us that there are homeomorphisms $\phi_n:C_n\to\sigma^2\times[0,1]$. Since $\sigma^2\times [0,1] $ is homeomorphic to $B_n$
this implies that there are homeomorphisms $\varphi_n:~B_n\to C_n$ which can also be arranged so that they fix the square bases of $B_n$.

Define now the map $H:S\to\mathbb{R}^3$ as 
\[H(x):=\begin{cases}
	\varphi_n(x),&\hspace{2mm}\text{for}\hspace{1mm}x\in B_n,\\
	x,&\hspace{2mm}\text{otherwise}.
\end{cases}\]

It is easy to see, by our construction, that this map is a homeomorphism of $S$ onto $H(S)$.

Suppose now that there exists a homeomorphism $h:\mathbb{R}^3\to\mathbb{R}^3$ such that $h(S)=H(S)$. Any homeomorphism that sends $S$ to $H(S)$ will map the base of $S$, $I^2\times\{0\}$ to the base of $H(S)$ and the hairs of $S$ to the hairs of $H(S)$. Let $\gamma(t)$, $t\in[0,1]$ be a parametrization of the longest hair of $S$ attached at $(1/2,1/2,0)$. By \cite[Theorem 4, Chapter 19]{Moise1977} and our construction we have that $H(\gamma(t))$ will be a wild arc and since it is a hair of $H(S)$ it will be the image of a hair of $S$ under $h$. However this is a contradiction since a wild arc cannot be the image of an arc under a homeomorphism of $\mathbb{R}^3$ like $h$.
\end{proof}

\bibliographystyle{plain}
\bibliography{bibliography}
\end{document}